\documentclass[a4paper,12pt]{article}
\usepackage{amsmath,fancyhdr,amssymb,amsthm, geometry, enumerate, graphicx ,amsfonts,hyperref,mathrsfs}

\author{Radhika Vasisht $^{1,\dag}$ and Ruchi Das $^2$}
\title{Induced Dynamics in Hyperspaces of Non-Autonomous Discrete Systems} 

\theoremstyle{definition}
\newtheorem{defn}{Definition}[section]
\providecommand{\keywords}[1]{\textbf{Keywords :} #1}
\providecommand{\msc}[1]{\textbf{MSC(2010)} #1}
\theoremstyle{plain}
\newtheorem{thm}{Theorem}[section]

\newtheorem{Cor}{Corollary}[section]
\newtheorem{exm}{Example}[section]

\newtheorem{lem}{Lemma}[section]

\begin{document}
\date{}
\maketitle

\begin{abstract}
In this paper, the interrelations of some dynamical properties of the non-autonomous dynamical system $(X,f_{1,\infty})$ and its induced non-autonomous dynamical system $(\mathcal{K}(X),\overline{f_{1,\infty}})$ are studied, where $\mathcal{K}(X)$ is the hyperspace of all non-empty compact subsets of $X$, endowed with Vietoris topology. Various stronger forms of sensitivity and transitivity are considered. Some examples of non-autonomous systems are provided to support the results. A relation between shadowing property of the non-autonomous system $(X,f_{1,\infty})$ and its induced system $(\mathcal{K}(X),\overline{f_{1,\infty}})$ is studied.
\\
\keywords{Non-autonomous dynamical systems, hyperspace, sensitivity, shadowing property}
\\ 
\msc{Primary 54H20; Secondary 37B55, 54B20}
\end{abstract}
\renewcommand{\thefootnote}{\fnsymbol{footnote}}
\footnotetext{\hspace*{-5mm}
\renewcommand{\arraystretch}{1}
\begin{tabular}{@{}r@{}p{15cm}@{}}
$^\dag$& the corresponding author. Email address: radhika.vasisht92@gmail.com (R. Vasisht)\\
$^1$&Department of Mathematics, University of Delhi, Delhi-110007, India\\
$^2$&Department of Mathematics, University of Delhi, Delhi-110007, India\\

\end{tabular}}
\section{Introduction}
Topological Dynamical System is one of the most applicable branches of mathematics devoted to the study of systems that are governed by uniform set of laws over time such as difference and differential equations. An autonomous discrete dynamical system is a dynamical system which has no external input and always evolves according to the same unchanging law. Most of the natural phenomenons are subjected to time-dependent external forces and their modeling leads to a mathematical theory of what are called non-autonomous discrete dynamical systems. The theory of non-autonomous dynamical systems helps characterizing the behaviour of various natural phenomenons which cannot be modeled by autonomous systems. The mathematical theory of non-autonomous systems is considerably more involved than the theory of autonomous systems. Non-autonomous discrete dynamical systems were introduced by authors in \cite{MR1402417}. Over recent years, the theory of such systems has developed into a highly active field related to, yet recognizably distinct from that of classical autonomous dynamical systems \cite{MR3779023, MR3528201, MR3516121, MR3717471, MR2922209, MR2996593}. Most of the natural phenomenon arise as a collection of several individual components and thus set valued dynamics are of great importance for studying any of these phenomenon. There are many applications of this approach in different branches of Science. Thus, there was a strong need to study the dynamical behaviour of induced spaces. Many researchers have worked in this direction. However, most of the study has been done when the system evolves according to the same unchanging law, but this approach fails to analyse the dynamics of the system governed by  the rules that change with time. So, the study of induced systems for non-autonomous dynamical systems is of utmost importance \cite{MR3584171}.
We first introduce some notations.
Consider the following non-autonomous discrete dynamical system (N D S) $(X,f_{1,\infty})$:
\begin{equation} x_{n+1}=f_n(x_n) , n \geq 1\nonumber \end{equation}
where $(X,d)$ is a compact metric space and $f_n: X \rightarrow X$ is a continuous map. For convenience, denote $f_{1,\infty} = ({f_n})_{n =1}^\infty $. Naturally, a difference equation of the form $x_{n+1}=f_n(x_n)$ can be thought of as the discrete analogue of a non-autonomous differential equation $\frac{dx}{dt}=f(x,t)$. 

Sensitive dependence on initial conditions or simply sensitivity, also known as the butterfly effect, is the main ingredient of chaos \cite{MR1963683}. In a system exhibiting sensitivity, a small change in the initial conditions will lead to a significant change in the dynamics of the system. Sensitivity analysis has a major application in the area of population biology \cite{MR2944184}. For continuous self maps of compact metric spaces, Moothathu \cite{MR2351026} gave an insight of the stronger forms of sensitivity and transitivity based on the largeness of subsets of $\mathbb{N}$. Since then several other stronger forms of both sensitivity and transitivity have been studied by different researchers. In \cite{MR2922208}, the author studies the relations between various forms of both sensitivity and transitivity  of the systems $(X,f)$ and $(\mathcal{K}(X),\overline{f})$, where $\mathcal{K}(X)$ denotes the hyperspace of all non-empty compact subsets of X. In \cite{MR3339062,MR3070943}, authors have studied various forms of sensitivity for product maps. Another important property in the computation of dynamical systems is the concept of shadowing \cite{MR2880453}. For a map $f$, $\delta$-pseudo-orbit is sequence (finite or infinite) of points such that the distance between $f(x_i)$ and $(x_{i+1})$ is less than $\delta$. A $\delta$-pseudo-orbit is said to be $\epsilon$-traced if there is a real point whose iterates track the pseudo-orbit within a distance of $\epsilon$,i.e, the pseudo-orbit is uniformly approximated by a genuine orbit. A map is said to have shadowing property if every $\delta$-pseudo orbit is $\epsilon$-traced. Shadowing has various applications in numerical analysis \cite{MR1858802}. In \cite{MR3570213}, authors have studied the relation between the shadowing property of the system $(X,f)$ and its induced hyperspace.
\\Motivated by the work discussed above for the induced systems of autonomous dynamical systems, we study such relations for non-autonomous systems. In Section 2, we give the preliminaries required for the remaining sections. In Section 3, we study the relations among various stronger forms of both sensitivity and transitivity for the non-autonomous system $(X,f_{1,\infty})$ and its induced systems $(\mathcal{K}(X),\overline{f_{1,\infty}})$. We also study various stronger forms of sensitivity for product maps. Further we give examples justifying our results. In Section 4, we establish a relation between the shadowing property of the non-autonomous system $(X,f_{1,\infty})$ and its induced system $(\mathcal{K}(X),\overline{f_{1,\infty}})$.
\section{Preliminaries}
In this section we recall some well known notions.

For any two open sets U and V of $X$, denote, 
$N_{f_{1,\infty}}(U,V)= \{n\in \mathbb{N} : f_1^n(U)\cap V\ne \emptyset \}$.
Let $ V\subset X$ be a non-empty open subset, $\mathbb{N}$ be the set of positive integers and $\delta>0$ . Denote 
$N_{f_{1,\infty}}(V,\delta) = \{n\in \mathbb{N}$ such that, there exist   $x,y \in V$ with $ d(f_1^n(x),f_1^n(y))> \delta \}$.

\begin{defn} A set $F\subset\mathbb{N}$ is called $\textit{syndetic}$ if there exists a positive integer a such that $\{i,i+1,.......,i+a\}\cap F\ne\emptyset$, for every $i\in\mathbb{N}$.
\end{defn}

\begin{defn}
A $\textit{thick set}$ is a set of integers that contains arbitrarily long intervals, that is, given a thick set $T$, for every $p\in \mathbb{N}$, there is some $n\in \mathbb{N}$  such that $\{n,n+1,n+2,...,n+p\}\subset T $.
\end{defn}

\begin{defn}
A set $F\subset\mathbb{N}$ is called $\textit{thickly syndetic}$ if $\{n\in\mathbb{N}: n+j\in F, 0\leq j\leq k\}$ is syndetic for each $k\in\mathbb{N}$. Then taking $n=a$ in the definition of syndetic set, we get that every thickly syndetic subset of $\mathbb{N}$ is syndetic.
\end{defn}

\begin{defn}
Let $|N_{f_{1,\infty}}(U,V)|$ be the cardinal number of the set $N_{f_{1,\infty}}(U,V)$ . Then 
\[\limsup_{n \to \infty} \frac{|N_{f_{1,\infty}}(U,V)\cap N_n|}{n}\] is called the upper density of $N_{f_{1,\infty}}(U,V)$, where $N_n=\{0,1,2,...,{n-1}\}$.

\end{defn}

\begin{defn}
The system $(X,f_{1,\infty})$ is said to be $\textit{topologically transitive}$ if for any two non-empty open sets $U_0$ and $V_0$ in $X$, there exists a positive integer $n\in \mathbb{N}$ such that, $U_n\cap V_0 \ne \emptyset$,
where $U_{i+1}=f_i(U_i)$, for every i, $ 1\leq i \leq n$ i.e. $f_{n}of_{n-1}o\cdots of_1(U_0)\cap V_0 \ne \emptyset$. Thus, system $(X,f_{1,\infty})$ is topologically transitive if for any two non-empty open sets $U_0$ and $V_0$ of $X$, $N_{f_{1,\infty}}(U_0,V_0)$ is non-empty.
\end{defn}

\begin{defn}
The system $(X,f_{1,\infty})$ is said to be $\textit{topologically mixing}$ if for any two non-empty open sets $U_0$ and $V_0$ in $X$, there exists a positive integer $N\in \mathbb{N}$ such that for any $n\geq N$, $U_n\cap V_0 \ne \emptyset$,
where $U_{i+1}=f_i(U_i)$, for every i, $ 1\leq i \leq n$ i.e. $f_{n}of_{n-1}o\cdots of_1(U_0)\cap V_0 \ne \emptyset $ ,for all $ n\geq N $. Thus, system $(X,f_{1,\infty})$ is topologically mixing if for any two non-empty open sets $U_0$ and $V_0$ of $X$ , there is a positive integer N such that $N_{f_{1,\infty}}(U_0,V_0) \supset [N,\infty)\cap\mathbb{N}$.
\end{defn}

\begin{defn}
The system $(X,f_{1,\infty})$ is said to be $\textit{syndetic transitive}$ if for any two non-empty open sets $U_0$ and $V_0$ in $X$, $N_{f_{1,\infty}}(U_0,V_0)$ is syndetic.
\end{defn}

\begin{defn}
The system $(X,f_{1,\infty})$ is said to be $\textit{topologically ergodic}$ if for any two non-empty open sets $U_0$ and $V_0$ in $X$, $N_{f_{1,\infty}}(U_0,V_0)$ has positive upper density.
\end{defn}

\begin{defn}
The system $(X,f_{1,\infty})$ is said to have $\textit {sensitive dependence on initial}$ $ \textit{conditions}$ if there exists a constant $ \delta_0>0$ such that for any $x_0 \in X$ and any neighbourhood U of $x_0$, there exists $y_0\in X\cap U$ and a positive integer n such that $d(x_n,y_n)>\delta_0$, where $\{x_i\}_{i=0}^\infty$ and  $\{y_i\}_{i=0}^\infty$ are the orbits of the system $(X,f_{1,\infty})$ starting from $x_0$ and $y_0$ respectively, the constant $\delta_0>0$ is called a sensitivity constant of the system $(X,f_{1,\infty})$.
Here  $({x_i})_{i=0}^\infty = \{x\in X$ such that $f_1^i(x) , i\geq 1\}$ where $f_1^i(x) = f_{i}\circ\cdots\circ f_1(x)$.
Then system $(X,f_{1,\infty})$ is said to have sensitive dependence on initial conditions or is $\textit{sensitive}$ in $X$ if there exists a constant $\delta>0$ such that for any non-empty open set V of $X$, $N_{f_{1,\infty}}(V,\delta)$ is non-empty.
\end{defn}

\begin{defn}
The system $(X,f_{1,\infty})$ is called $\textit{cofinitely sensitive}$ in $X$ if there exists a constant $\delta>0$ such that for any non-empty open set V of $X$, there exists $N\geq1$ such that $[N,\infty)\cap\mathbb{N} \subset N_{f_{1,\infty}}(V,\delta)$; $\delta$ is called a constant of cofinite sensitivity.
\end{defn}
  
\begin{defn}
The system $(X,f_{1,\infty})$ is said to have $\textit{syndetic sensitivity}$ in $X$ if there exists a constant $\delta>0$ such that for any non-empty open set V of $X$, $N_{f_{1,\infty}}(V,\delta)$ is syndetic; $\delta$ is called a constant of syndetic sensitivity.
\end{defn}

\begin{defn} 
The system $(X,f_{1,\infty})$ is said to be $\textit{thickly syndetic sensitive}$ in $X$ if there exists a constant $\delta>0$ such that for any non-empty open set V of $X$, $N_{f_{1,\infty}}(V,\delta)$ is thickly syndetic; $\delta$ is called a constant of thickly syndetic sensitivity \cite{MR3528201}.
\end{defn}

We have, \bigskip

cofinitely sensitive $\implies$ thickly syndetic sensitive $\implies$ syndetic sensitive $\implies$ sensitive.
\\The following notions (Definitions 2.13-2.15) have been defined by us in \cite{vasisht2018}.

\begin{defn}
The system $(X,f_{1,\infty})$ is said to have $\textit{thick sensitivity}$ in $X$ if there exists a constant $\delta>0$ such that for any non-empty open set V of $X$, $N_{f_{1,\infty}}(V,\delta)$ is thick; $\delta$ is called a constant of thick sensitivity.
\end{defn}

\begin{defn}
The system $(X,f_{1,\infty})$ is said to have $\textit{ergodic sensitivity}$ in $X$ if there exists a constant $\delta>0$ such that for any non-empty open set V of $X$, $N_{f_{1,\infty}}(V,\delta)$ has positive upper density; $\delta$ is called a constant of ergodic sensitivity.
\end{defn}

\begin{defn}
The system $(X,f_{1,\infty})$ is said to have $\textit{multi-sensitivity}$ in $X$ if there exists a constant $\delta>0$ such that for  every $k\geq 1$ and any non-empty open subsets $V_{1},V_{2},\ldots ,V_{k}$ of $X$, $\cap_{i=0}^k \mathbb{N}_{f_{1,\infty}}(V_{i},\delta)$ is non-empty; $\delta$ is called a constant of multi-sensitivity.
\end{defn}

\begin{defn}
A finite or infinite sequence $\{x_0, x_1, x_2, \ldots \} \subseteq X$, is a $\delta \textit{-pseudo orbit}$, for some $\delta>0$, if $d(f_i(x_{i-1}),x_i)<\delta$, for all $i\geq1$ \cite{MR3174280}.
\end{defn}

\begin{defn}
The system $(X,f_{1,\infty})$ is said to have $\textit{shadowing property}$ if for every $\epsilon>0$, there exists $\delta=\delta(\epsilon)>0$ such that for every $\delta$-pseudo orbit $\{x_0, x_1, x_2,\ldots\}\subseteq X$, there is a $y\in X$ such that for all $i\geq 0$, $d(f_{0}^{i}(y),x_i)<\epsilon$ \cite{MR3174280, MR3206430}.
\end{defn}

\begin{defn}
The system $(X,f_{1,\infty})$ is said to have $\textit{finite-shadowing property}$ if for every $\epsilon>0$, there exists $\delta=\delta(\epsilon)>0$ such that for every finite $\delta$-pseudo orbit $\{x_0, x_1, x_2,\ldots, x_n\}$ $\subseteq X$, there is a $y\in X$ such that for all $0\leq i\leq n$, $d(f_{0}^{i}(y),x_i)<\epsilon$.
\end{defn}

Let $X$ be a topological space. Then $\mathcal{K}$($X)$ denotes the hyperspace of all non-empty compact subsets of $X$ endowed with the \textit{Vietoris Topology}. A basis of open sets for Vietoris topology is given by following sets:
\bigskip

\noindent $< U_1, U_2, \ldots , U_k >$ = $\{K \in$ $\mathcal{K}$($X)$: $K \subset \bigcup_{i=1}^{k} U_{i}$ and $K\cap U_{i}$  $\ne \emptyset$, for each $i$ $\in \{1, 2, \ldots ,k\}$\},

\bigskip\noindent where $U_1, U_2, \ldots ,U_k$ are non-empty open subsets of $X$.

Given metric space $(X, d)$, a point $x \in X$ and $A \in \mathcal{K}(X)$, let $d(x, A)$ = $\inf \{d(x, a): a \in A \}$. For every $\epsilon > 0$, let  open $d$-ball in $X$ about $A$ and radius $\epsilon$ be given by  $B_{d}(A, \epsilon) = \{x \in X: d(x, A) < \epsilon\} = \bigcup_{a \in A} B_d(a, \epsilon)$ where $B_d(a, \epsilon)$ denotes the open ball in $X$ centred at $a$ and of radius $\epsilon$. The Hausdorff metric on $\mathcal{K}$($X)$  induced by $d$, denoted by $d_H$, is defined  as follows:

\ \ \ \ \ \ \ \ \ \ \ \ \ \ \ \ \ $d_H$($A, B) = \inf \{ \epsilon > 0: A \subseteq B_{d}(B, \epsilon) \  \text{and} \  B \subseteq B_{d}(A, \epsilon)\}$,

\bigskip\noindent where $A$, $B \in$ $\mathcal{K}$($X)$. We shall recall that the topology induced by the Hausdorff metric coincides with the Vietoris topology if and only if the space $X$ is compact. Also, for a compact metric space $X$ and  $A, B \in$ $\mathcal{K}$($X)$, we get that $d_H$($A, B) < \epsilon$ if and only if $A \subseteq B_d(B, \epsilon)$ and $B \subseteq B_d(A, \epsilon)$.
\\Let $\mathcal{F}(X)$ denote the set of all finite subsets of $X$. Under Vietoris topology, $\mathcal{F}(X)$ is dense in $\mathcal{K}(X)$ \cite{MR2665229, MR1269778}.
 Given a continuous function $f: X \to X$, it induces a continuous function $\overline{f}$: $\mathcal{K}$($X) \to$ $\mathcal{K}$($X)$ defined by $\overline{f}(K) = f(K)$, for every $K \in$ $\mathcal{K}$($X)$, where $f(K)$ = $\{f(k) : k\in K\}$. Note that continuity of $f$ implies continuity of $\overline{f}$.

\bigskip Let $(X, f_{1,\infty})$ be a non-autonomous discrete dynamical system and $\overline{f}_n$ be the function on $\mathcal{K}$($X)$, induced by $f_n$ on $X$, for every $n\in\mathbb{N}$. Then the sequence $\overline{f}_{1,\infty}$ = ($\overline{f}_1, \overline{f}_2$, $\ldots ,\overline{f}_n, \ldots )$ induces a non-autonomous discrete dynamical system ($\mathcal{K}$($X), \overline{f}_{1,\infty})$ and here $\overline{f}_1^n = \overline{f}_n \circ \ldots \circ \overline{f}_2\circ \overline{f}_1$.  Note that $\overline{f}_1^n = \overline{f^n_1}$.

\bigskip Let $(X,d_X)$ and $(Y,d_Y)$ be compact metric spaces. For non-autonomous discrete dynamical systems $(X,f_{1,\infty})$ and $(Y,g_{1,\infty})$, put $(f_{1,\infty}\times g_{1,\infty})= (h_{1,\infty}) = (h_1,h_2,\ldots,h_n,\ldots)$, where $h_n=f_n\times g_n$ , for each $n\in \mathbb{N}$. Thus, $(X\times Y, f_{1,\infty}\times g_{1,\infty})$ is a non-autonomous dynamical system, where $(X\times Y)$ is a compact metric space endowed with the product metric $d_{X\times Y}((x,y),(x',y'))= d_X(x,x')+d_Y(y,y')$. Here, $h_{1}^{n}=h_n\circ h_{n-1}\circ \cdots \circ h_2\circ h_1 = (f_n\times g_n)\circ (f_{n-1}\times g_{n-1})\circ \cdots \circ (f_2\times g_2)\circ (f_1\times g_1)$ \cite{MR3584171}.
\\
\\We shall use the following result.
\begin{lem}
Let $a,b,c,d$ be real numbers with $a<b$ and $c<d$. If there is an $L>0$ such that $(b-a)\leq L$ and $(d-c)\leq L$, then $min\{b,d\}-min\{a,c\}\leq L$ \cite{MR2922208}.
\end{lem}

\section{On Various Stronger forms of Sensitivity and Transitivity}
  
In this Section, we give the interrelations of various stronger forms of sensitivity and transitivity of the non-autonomous dynamical system $(X,f_{1,\infty})$ and its induced system $(\mathcal{K}(X),\overline{f_{1,\infty}})$. We provide two examples of non-autonomous systems to support our results.

\begin{thm}
The dynamical system $(X,f_{1,\infty})$ is syndetic sensitive if and only if induced system $(\mathcal{K}$($X),\overline{f}_{1,\infty}$) is syndetic sensitive.
\end{thm}

\begin{proof}
Let $(X,f_{1,\infty})$ be syndetic sensitive with constant $\delta > 0$. Since $\mathcal{F}$($X)$ is dense in $\mathcal{K}$($X)$, it suffices to prove the result for $f_{1,\infty}|_{\mathcal{F}(X)}$.
Let $A$ = \{$x_i:1\leq i\leq k\} \in \mathcal{F}(X)$ and $B_{d_H}(A,\epsilon)$ be an $\epsilon$-neighbourhood of $A$ and $B_d(x_i,\epsilon)$ be an $\epsilon$-neighbourhood of $x_i$ for each $1 \leq i \leq k$. Write \{$n \in \mathbb{N}:\underset{y\in B_d(x_i,\epsilon)} {sup}$ $d(f_{1}^{n}(x_i),f_{1}^{n}(y))>\delta\} =\{n(i,j): n(i,j+1)>n(i,j);j \in \mathbb{N}\}$, for each i, $1\leq i \leq k$.

Since $(X,f_{1,\infty})$ is syndetic sensitive, therefore for each $1\leq i\leq k$, there exists an $L_i$ such that $n(i,j+1)-n(i,j)\leq L_i$, for all $j \in \mathbb{N}$.
Let $L=max\{L_i:1\leq i\leq k\}$. Then for each $x_i, (1\leq i\leq k)$, there exists $y_i \in B_{d}(x_i,\epsilon)$ and $0\leq r_i\leq L$ with $d(f_{1}^{r}(x_i),f_{1}^{r}(y_{i}))>\delta$, $1\leq i\leq k$.

Let $r=min\{r_i:1\leq i\leq k\}$. Since each $f_{n}, n\in \mathbb{N}$ is continuous, therefore, $f_{1}^{n}$ is continuous , for all $n \in \mathbb{N}$ and hence $X$ being compact, $f_{1}^{n}$ is uniformly continuous  for each $n$. Thus, $f_{1}^{i}$ is uniformly continuous for each i, $0\leq i\leq L$ and hence there exists $\delta_0$, $0\leq \delta_{0}\leq \delta$ such that $d(f_{1}^{r}(x_{i}),f_{1}^{r}(y_{i})>\delta$ for each i, $1\leq i\leq k$.
We take $C= \{z_i:1\leq i\leq k\}$ such that following conditions hold

\begin{enumerate}
\item If $d(f_{1}^{r}(x_{1}),f_{1}^{r}(x_{i}))\leq \delta_{0}/2$, then $z_i = y_i$; 
\item If $d(f_{1}^{r}(x_{1}),f_{1}^{r}(x_{i}))> \delta_{0}/2$, then $z_i = x_i$. 
\end{enumerate}
Therefore, $d(f_{1}^{r}(x_{1}),f_{1}^{r}(z_{i}))> \delta_{0}/2$, for each i, $1\leq i\leq k$. Consequently, $d_{H}(\overline{f_{1}^{r}}(A),\overline{f_{1}^{r}}(C))$ $>\delta_{0}/2$. Let $n_{j}=min\{n(i,j):1\leq i\leq k\}$ , for all $j\geq 0$. Since \{$n \in \mathbb{N}: sup$ $d(f_{1}^{n}(x_i),f_{1}^{n}(y))>\delta\} =\{n(i,j): n(i,j+1)>n(i,j);j \in \mathbb{N}\}$ is syndetic with $n(i,j+1)-n(i,j)\leq L_{i}<L$ , for all $j\in \mathbb{N}$ and for each i, $1\leq i\leq k$. Therefore, by lemma 2.1, $\{n_j :j\geq0\}$ is also syndetic with $n_{j+1}-n_{j}<L$. Hence, $N_{\overline{f_{1,\infty}}}(B_{d_H}(A,\epsilon),\delta/2)$ is syndetic. Thus, $(\mathcal{K}$($X),\overline{f}_{1,\infty}$) is syndetic sensitive.

Conversely, suppose that $(\mathcal{K}$($X),\overline{f}_{1,\infty}$) is syndetic sensitive with constant of syndetic sensitivity $\delta>0$. For any $\epsilon>0$, let $x\in X$ and $U$ be the $\epsilon$- neighbourhood of $x$ in $X$. Since $B_{d_H}(\{x\},\epsilon)$ is an $\epsilon$- neighbourhood of $\{x\}$ in $\mathcal{K}$($X)$ and we know $\overline{f_{1}^{\infty}}$ is syndetic sensitive, so $N_{\overline{f_{1,\infty}}}[{B_{d_H}(\{x\},\epsilon),\delta)}]$ is syndetic and therefore there exist $A\in B_{d_H}(\{x\},\epsilon)$ and $n\geq 0$ such that $d_{H}(\overline{f_{1}^{n}}(\{x\}),\overline{f_{1}^{n}}(A))>\delta$.

Hence, there exists $y\in A \subset U$ such that $d(f_{1}^{n}(x),f_{1}^{n}(y))>\delta$ which implies $N_{\overline{f_{1,\infty}}}$ $[\cup{B_{d_H}(\{x\},\epsilon),\delta)};x\in U] \subset N_{f_{1,\infty}}(U,\delta)$. Since $N_{\overline{f_{1,\infty}}}[\cup{B_{d_H}(\{x\},\epsilon),\delta)};x\in U]$ is syndetic, therefore $N_{f_{1,\infty}}(U,\delta)$ is syndetic. Hence, $(X,f_{1,\infty})$ is syndetic sensitive.
\end{proof}

\begin{thm}
Let $(X,f_{1,\infty})$ and $(Y,g_{1,\infty})$ be two dynamical systems. If $(X,f_{1,\infty})$ or $(Y,g_{1,\infty})$ is syndetic sensitive, then $(X\times Y,f_{1,\infty}\times g_{1,\infty})$ is syndetic sensitive.
\end{thm}

\begin{proof}
Suppose $(X,f_{1,\infty})$ is syndetic sensitive with constant of syndetic sensitivity $\delta>0$. Let $U\times V$ be a non-empty open set in $X\times Y$. Then, $U$ is a non-empty open set in $X$, therefore by syndetic sensitivity of $(X,f_{1,\infty})$, we have that $N_{f_{1,\infty}}(U,\delta)$ is syndetic. Since $N_{f_{1,\infty}}(U,\delta)\cup N_{g_{1,\infty}}(V,\delta)\subset N_{f_{1,\infty}\times g_{1,\infty}}(U\times V,\delta)$, therefore $N_{f_{1,\infty}\times g_{1,\infty}}(U\times V,\delta)$ is also syndetic. Thus, $(X\times Y,f_{1,\infty}\times g_{1,\infty})$ is syndetic sensitive. 
Similarly, the result holds when $(Y,g_{1,\infty})$ is syndetic sensitive.
\end{proof}

\begin{Cor}
Let $(X,f_{1,\infty})$ and $(Y,g_{1,\infty})$ be two dynamical systems. If $(X,f_{1,\infty})$ or $(Y,g_{1,\infty})$ is syndetic sensitive, then $(\mathcal{K}$($X\times Y),\overline{(f\times g)}_{1,\infty}$) is syndetic sensitive.
\end{Cor}

\begin{proof}
The proof follows from Theorem 3.1 and Theorem 3.2.
\end{proof}

\begin{exm} \end{exm}
Let $I$ be the interval $[0,1]$ and $f$ on $I$ be defined by: 
\[ f(x) = \begin{cases}
1/2-2x,  & \text{for} \ x \in \left[0, \frac{1}{4}\right] \\
4x-1, & \text{for} \ x \in \left[\frac{1}{4},\frac{1}{2}\right] \\
2-2x, & \text{for} \ x \in \left[\frac{1}{2},1\right].
\end{cases} \]

Let $f_{2n}(x)=x$, for all $x$ in [0,1] and $f_{2n-1}(x)=f(x)$, for all $n \in \mathbb{N}$. Since $f(x)$ is transitive on $I$, therefore it is cofinitely sensitive\cite{MR2351026} and hence syndetic sensitive. Hence, we can say that the non-autonomous system $(I,f_{1,\infty})$ is syndetic sensitive. Thus, by Theorem 3.1, the induced system $(\mathcal{K})$($I),\overline{f}_{1,\infty})$ is syndetic sensitive. \\Also, let $(Y,g_{1,\infty})$ be any non-autonomous system, then by Theorem 3.2 and Corollary 3.1, we get that systems $(I\times Y,f_{1,\infty}\times g_{1,\infty})$ and $(\mathcal{K}$($I\times Y),\overline{(f\times g)}_{1,\infty}$) are both syndetic sensitive.

\begin{thm}
The dynamical system $(X,f_{1,\infty})$ is multi-sensitive if and only if $(\mathcal{K}$($X),\overline{f}_{1,\infty}$) is multi-sensitive.
\end{thm}

\begin{proof}
Let $(X,f_{1,\infty})$ be multi-sensitive with constant $\delta > 0$. Since $\mathcal{F}$($X)$ is dense in $\mathcal{K}$($X)$, it suffices to prove the result for $f_{1,\infty}|_{\mathcal{F}(X)}$.
For $k\geq 1$, let $A_j$ = \{$x_{j,i}:1\leq i\leq k_j\} \in \mathcal{F}(X)$. Let $B_{d_H}(A_j,\epsilon)$ be the $\epsilon$-neighbourhood of $A_j$ and $B_d(x_{j,i},\epsilon)$ be the $\epsilon$-neighbourhood of $x_{j,i}$ for each i, $1 \leq i \leq k_j$. Since $(X,f_{1,\infty})$ is multi-sensitive, for each j, $1\leq j\leq k$ and for each i, $1\leq i\leq k_{j}$, there exists $n>0$ such that $\underset{y \in B_d(x_{j,i},\epsilon)}{sup}d(f_{1}^{n}(x_{j,i}),f_{1}^{n}(y))>\delta$ for every j, $1\leq j\leq k$ and for each i, $1\leq i\leq k_j$. We shall show that $\underset{B\in B_{d_H}(A,\epsilon)}{sup}d(\overline{f_{1}^{n}}(A_j),\overline{f_{1}^{n}}(B))>\delta/2$, for every j, $1\leq j\leq k$. By definition of multi sensitivity and from the above argument for each $x_{j,i}$, there exists $y_{j,i}\in B_d(x_{j,i},\epsilon)$ such that $d(f_{1}^{n}(x_{j,i}),f_{1}^{n}(y_{j,i}))>\delta$.
For each j, $1\leq j\leq k$,take $C_j=\{z_{j,1},z_{j,2},\ldots,z_{j,k_j}\}$ such that the following conditions hold
\begin{enumerate}
\item If $d(f_{1}^{n}(x_{j,1}),f_{1}^{n}(x_{j,i}))\leq \delta/2$, then $z_{j,i} = y_{j,i}$; 
\item If $d(f_{1}^{n}(x_{j,1}),f_{1}^{n}(x_{j,i}))> \delta/2$, then $z_{j,i}= x_{j,i}$.
\end{enumerate}
Therefore, $d(f_{1}^{n}(x_{j,1}),f_{1}^{n}(z_{j,i})> \delta/2$ , for all j, $1\leq j\leq k$ and for all i, $1\leq i\leq k_j$.
Consequently, $d_H(\overline{f_{1}^{n}}(A_j),\overline{f_{1}^{n}}(C_j))>\delta/2$. Therefore, $\underset{1\leq j\leq k}{\cap}N_{\overline{f_{1,\infty}}}((A_j,\epsilon),\delta/2)$ is non-empty for any $k\geq 1$ and any $\epsilon>0$. 
Thus, $(\mathcal{K}$($X),\overline{f}_{1,\infty}$) is multi-sensitive.

Conversely, assume that $(\mathcal{K}$($X),\overline{f}_{1,\infty}$) is multi-sensitive with constant of multi sensitivity $\delta>0$. For any $\epsilon>0$ and any $k\geq 1$, let $x_i \in X$ and $U_i$ be the $\epsilon$-neighbourhood of $x_i$, for each i, $1\leq i\leq k$ respectively. Since $B_{d_H}(\{x_i\},\epsilon)$ is an open $\epsilon$-neighbourhood of $\{x_i\}$ in $(\mathcal{K}$($X)$ and $\overline{f_{1}^{\infty}}$ is multi-sensitive, therefore $\underset{1\leq i\leq k}{\cap}N_{\overline{f_{1,\infty}}}(B_{d_H}\{x_i\},\epsilon)$ is non-empty. Let $m\in \underset{1\leq i\leq k}{\cap}N_{\overline{f_{1,\infty}}}(B_{d_H}\{x_i\},\epsilon)$, then for each i, $1\leq i\leq k$, there exists $A_i \in B_{d_H}(\{x_i\},\epsilon)$ such that $d_H(\overline{f_{1}^{m}}(\{x_i\}),\overline{f_{1}^{m}}(\{A_i\}))>\delta$. Therefore, there exists $y_i \in A_i$ such that $d(f_{1}^{m}(x_i),f_{1}^{m}(y_i))>\delta$ , for all i, $1\leq i\leq k$. Hence, $m\in N_{f_{1,\infty}}(U_i,\delta)$ , for all i, $1\leq i\leq k$. Thus $\underset{1\leq i\leq k}{\cap}N_{f_{1,\infty}}(U_i,\delta)$ is non empty implying $(X,f_{1,\infty})$ is multi-sensitive.
\end{proof}
We recall the following result (\cite{MR2}, Theorem 3.1).\\ 
\begin{thm}
Let $(X,f_{1,\infty})$ and $(Y,g_{1,\infty})$ be two dynamical systems. The system $(X\times Y,f_{1,\infty}\times g_{1,\infty})$ is multi-sensitive if and only if $(X,f_{1,\infty})$ or $(Y,g_{1,\infty})$ is multi-sensitive.
\end{thm}

Based on Theorem 3.3 and Theorem 3.4, we have the following Corollary.

\begin{Cor}
Let $(X,f_{1,\infty})$ and $(Y,g_{1,\infty})$ be two dynamical systems. Then the system $(\mathcal{K}$($X\times Y),\overline{(f\times g)}_{1,\infty}$) is multi-sensitive if and only if $(X,f_{1,\infty})$ or $(Y,g_{1,\infty})$ is multi-sensitive.
\end{Cor}

\begin{exm} \end{exm}
Let $I$ be the interval $[0,1]$ and $f$ on $I$ be defined by: 

\[ f(x) = \begin{cases}
2x+1/2,  & \text{for} \ x \in \left[0, \frac{1}{4}\right] \\
-2x+3/2, & \text{for} \ x \in \left[\frac{1}{4},\frac{3}{4}\right] \\
2x-3/2, & \text{for} \ x \in \left[\frac{3}{4},1\right].
\end{cases} \]
Let $f_{2n}(x)=x$, for all $x$ in [0,1] and $f_{2n-1}(x)=f(x)$, for all $n \in \mathbb{N}$. Clearly, the autonomous system $(I,f)$ is sensitive and thus cofinitely sensitive \cite{MR2351026}. Thus, we can say that $(I,f_{1,\infty})$ is also cofinitely sensitive and hence multi-sensitive. So, by Theorem 3.3 the induced system $(\mathcal{K})$($I),\overline{f}_{1,\infty})$ is multi-sensitive. \\Also, let $(Y,g_{1,\infty})$ be any non-autonomous system, then by Theorem 3.4 and Corollary 3.2, we get that systems $(I\times Y,f_{1,\infty}\times g_{1,\infty})$ and $(\mathcal{K}$($I\times Y),\overline{(f\times g)}_{1,\infty}$) are both multi-sensitive.
\begin{thm}
Let $(X,f_{1,\infty})$ be a dynamical system. If $(\mathcal{K}$($X),\overline{f}_{1,\infty}$) is ergodically sensitive, then so is $(X,f_{1,\infty})$.
\end{thm}

\begin{proof}
Assume that $(\mathcal{K}$($X),\overline{f}_{1,\infty}$) is ergodically sensitive with constant of ergodic sensitivity $\delta>0$. For any $\epsilon>0$, let $x\in X$ and $U=B_d(x,\epsilon)$ be the $\epsilon$- neighbourhood of $x$ in $X$. Since $B_{d_H}(\{x\},\epsilon)$ is the $\epsilon$- neighbourhood of $\{x\}$ in $\mathcal{K}$($X)$ and we know that $\overline{f_{1}^{\infty}}$ is ergodically sensitive, therefore $N_{\overline{f_{1,\infty}}}[{B_{d_H}(\{x\},\epsilon),\delta)}]$ has positive upper density and hence there exist $A\in B_{d_H}(\{x\},\epsilon)$ and $n\geq 0$ such that $d_{H}(\overline{f_{1}^{n}}(\{x\}),\overline{f_{1}^{n}}(A))>\delta$.
\\Therefore, there exists $y\in A \subset U$ such that $d(f_{1}^{n}(x),f_{1}^{n}(y))>\delta$ which implies $N_{\overline{f_{1,\infty}}}$ $[\cup{B_{d_H}(\{x\},\epsilon),\delta)};x\in U] \subset N_{f_{1,\infty}}(U,\delta)$. As $N_{\overline{f_{1,\infty}}}[\cup{B_{d_H}(\{x\},\epsilon),\delta)};x\in U]$ has positive upper density, therefore $N_{f_{1,\infty}}(U,\delta)$ has positive upper density. Hence, $(X,f_{1,\infty})$ is ergodically sensitive.
\end{proof}

\begin{thm}
Let $(X,f_{1,\infty})$ and $(Y,g_{1,\infty})$ be two dynamical systems. The system $(X\times Y,f_{1,\infty}\times g_{1,\infty})$ is ergodically sensitive if and only if $(X,f_{1,\infty})$ or $(Y,g_{1,\infty})$ is ergodically sensitive.
\end{thm}

\begin{proof}
Suppose $(X,f_{1,\infty})$ is ergodically sensitive with constant of ergodic sensitivity $\delta>0$. Let $U\times V$ be a non-empty open set in $X\times Y$. Then, $U$ is a non-empty open set in $X$, so by ergodic sensitivity of $(X,f_{1,\infty})$, we have that $N_{f_{1,\infty}}(U,\delta)$ has positive upper density. Since $N_{f_{1,\infty}}(U,\delta)\cup N_{g_{1,\infty}}(V,\delta)\subset N_{f_{1,\infty}\times g_{1,\infty}}(U\times V,\delta)$ therefore $N_{f_{1,\infty}\times g_{1,\infty}}(U\times V,\delta)$ also has positive upper density. Thus, $(X\times Y,f_{1,\infty}\times g_{1,\infty})$ is ergodically sensitive. 
Similarly the result holds when $(Y,g_{1,\infty})$ is ergodic sensitive.

Conversely, suppose that $(X\times Y,f_{1,\infty}\times g_{1,\infty})$ is ergodically sensitive with constanst of ergodic sensitivity $\delta>0$. Let us assume that both $f_{1,\infty}$ and $g_{1,\infty}$ are not ergodically sensitive which implies that for any $\epsilon>0$, there exists an open set $U\subset X$ such that $\overline{d}(N_{f_{1,\infty}}(U,\epsilon))=0$ and there exists an open set $V\subset Y$ such that $\overline{d}(N_{g_{1,\infty}}(V,\epsilon))=0$. Thus, for $\epsilon=\delta/3$, there exist $U'\subset X$ and $V'\subset Y$ such that $\overline{d}(N_{f_{1,\infty}}(U',\delta/3))=0$ and $\overline{d}(N_{g_{1,\infty}}(V',\delta/3))=0$. Clearly, $N_{f_{1,\infty}\times g_{1,\infty}}(U'\times V',\delta)\subset N_{f_{1,\infty}}(U',\delta/3)\cup N_{g_{1,\infty}}(V',\delta/3)$. Therefore, \begin{align}
\overline{d}(N_{f_{1,\infty}\times g_{1,\infty}}(U'\times V',\delta))\leq \overline{d}(N_{f_{1,\infty}}(U',\delta/3)\cup N_{g_{1,\infty}}(V',\delta/3))\nonumber\end{align} \begin{align}\leq \overline{d}(N_{f_{1,\infty}}(U',\delta/3))+\overline{d}(N_{g_{1,\infty}}(V',\delta/3))=0
\nonumber\end{align}
which contradicts the ergodic sensitivity of $(X\times Y,f_{1,\infty}\times g_{1,\infty})$ and hence we have that $(X,f_{1,\infty})$ or $(Y,g_{1,\infty})$ is ergodically sensitive.
\end{proof}

From Theorem 3.5 and Theorem 3.6, we get that
\begin{Cor}
Let $(X,f_{1,\infty})$ be a dynamical system. If $(\mathcal{K}$($X\times Y),\overline{(f\times g)}_{1,\infty}$) is ergodically sensitive, then $(X,f_{1,\infty})$ or $(Y,g_{1,\infty})$ is ergodically sensitive.
\end{Cor}

\begin{thm}
Let $(X,f_{1,\infty})$ be a dynamical system. If $(\mathcal{K}$($X),\overline{f}_{1,\infty}$) is thickly sensitive or thickly syndetic sensitive, then so is $(X,f_{1,\infty})$.
\end{thm}

\begin{proof}
Let $(\mathcal{K}$($X),\overline{f}_{1,\infty}$) be thickly sensitive with constant of thick sensitivity $\delta>0$. For any $\epsilon>0$ and $x\in X$ let $U=B_d(x,\epsilon)$. Since $B_{d_H}(\{x\},\epsilon)$ is an $\epsilon$- neighbourhood of $\{x\}$ in $\mathcal{K}$($X)$ and we know $\overline{f_{1}^{\infty}}$ is thickly sensitive, so $N_{\overline{f_{1,\infty}}}[{B_{d_H}(\{x\},\epsilon),\delta)}]$ is thick therefore there exist $A\in B_{d_H}(\{x\},\epsilon)$ and $n\geq 0$ such that $d_{H}(\overline{f_{1}^{n}}(\{x\}),\overline{f_{1}^{n}}(A))>\delta$.
\\Hence, there exists $y\in A \subset U$ such that $d(f_{1}^{n}(x),f_{1}^{n}(y))>\delta$ which implies $N_{\overline{f_{1,\infty}}}$ $[\cup{B_{d_H}(\{x\},\epsilon),\delta)};x\in U] \subset N_{f_{1,\infty}}(U,\delta)$. As $N_{\overline{f_{1,\infty}}}[\cup{B_{d_H}(\{x\},\epsilon),\delta)};x\in U]$ is thick, therefore $N_{f_{1,\infty}}(U,\delta)$ is thick. Hence, $(X,f_{1,\infty})$ is thickly sensitive.
\\Similarly, one can prove when $(\mathcal{K}$($X),\overline{f}_{1,\infty}$) is thickly syndetic sensitive. 
\end{proof}

\begin{thm}
Let $(X,f_{1,\infty})$ and $(Y,g_{1,\infty})$ be two dynamical systems. If $(X,f_{1,\infty})$ or $(Y,g_{1,\infty})$ is thick sensitive or thickly syndetic sensitive, then so is $(X\times Y,f_{1,\infty}\times g_{1,\infty})$.
\end{thm}

\begin{proof}
Suppose $(X,f_{1,\infty})$ is thick sensitive with constant of thick sensitivity $\delta>0$. Let $U\times V$ be a non-empty open set in $X\times Y$. Then, $U$ is a non-empty open set in $X$, therefore by thick sensitivity of $(X,f_{1,\infty})$, we have that $N_{f_{1,\infty}}(U,\delta)$ is thick. Since $N_{f_{1,\infty}}(U,\delta)\cup N_{g_{1,\infty}}(V,\delta)\subset N_{f_{1,\infty}\times g_{1,\infty}}(U\times V,\delta)$, therefore, $N_{f_{1,\infty}\times g_{1,\infty}}(U\times V,\delta)$ is also thick. Thus, $(X\times Y,f_{1,\infty}\times g_{1,\infty})$ is thick sensitive. 
Similarly, the result holds when $(Y,g_{1,\infty})$ is thick sensitive.
\\By similar arguments, one can prove for $(X,f_{1,\infty})$ or $(Y,g_{1,\infty})$ being thickly syndetic sensitive.
\end{proof}

In \cite{MR3584171}, authors have proved results relating the transitivity of the non-autonomous system $(X,f_{1,\infty})$ and of its induced hyperspace $(\mathcal{K}(X),\overline{f_{1,\infty}})$. In the following results, we prove such relations for stronger forms of transitivity.

\begin{thm}
Let $(X,f_{1,\infty})$ be a dynamical system. If $(\mathcal{K}$($X),\overline{f}_{1,\infty}$) is syndetic transitive, then so is $(X,f_{1,\infty})$.
\end{thm}

\begin{proof}
Let $U$ and $V$ be two non-empty open sets in $X$, then $\mathcal{U}$=$<U>$ and $\mathcal{V}$=$<V>$ are non-empty open sets in  $(\mathcal{K}$($X)$). Since $(\mathcal{K}$($X),\overline{f}_{1,\infty}$) is syndetic transitive, therefore $N_{\overline{f_{1}^{\infty}}}(\mathcal{U},\mathcal{V})$ is syndetic. Let $n\in N_{\overline{f_{1}^{\infty}}}(\mathcal{U},\mathcal{V})$, so $\overline{f_{1}^{n}}(\mathcal{U},\mathcal{V})$ is non-empty. Then, there exists $K\in\mathcal{U}$ such that $\overline{f_{1}^{n}}(K)\in\mathcal{V}$ which implies there exists $x\in K\subset U$ such that $f_{1}^{n}(x)\in V$. Therefore, we have $n\in N_{f_{1,\infty}}(U,V)$ and hence $N_{\overline{f_{1}^{\infty}}}(\mathcal{U},\mathcal{V}$)$\subseteq N_{f_{1,\infty}}(U,V)$. Since $N_{\overline{f_{1}^{\infty}}}(\mathcal{U},\mathcal{V})$ is syndetic, therefore $N_{f_{1,\infty}}(U,V)$ is syndetic. Hence, $(X,f_{1,\infty})$ is syndetic transitive.
\end{proof}

\begin{thm}
Let $(X,f_{1,\infty})$ be a dynamical system. If $(\mathcal{K}$($X),\overline{f}_{1,\infty}$) is topologically ergodic, then so is $(X,f_{1,\infty})$.
\end{thm}

\begin{proof}
Let $U$ and $V$ be two non-empty open sets in $X$, then $\mathcal{U}$=$<U>$ and $\mathcal{V}$=$<V>$ are non-empty open sets in  $(\mathcal{K}$($X)$). Since $(\mathcal{K}$($X),\overline{f}_{1,\infty}$) is topologically ergodic, therefore $N_{\overline{f_{1}^{\infty}}}(\mathcal{U},\mathcal{V})$ has positive upper density. Let $n\in N_{\overline{f_{1}^{\infty}}}(\mathcal{U},\mathcal{V})$, then $\overline{f_{1}^{n}}(\mathcal{U},\mathcal{V})$ is non-empty. Therefore, there exists $K\in\mathcal{U}$ such that $\overline{f_{1}^{n}}(K)\in\mathcal{V}$ which implies there exists $x\in K\subset U$ such that $f_{1}^{n}(x)\in V$. Hence, we have $n\in N_{f_{1,\infty}}(U,V)$ implying $N_{\overline{f_{1}^{\infty}}}(\mathcal{U},\mathcal{V}$)$\subseteq N_{f_{1,\infty}}(U,V)$. Since $N_{\overline{f_{1}^{\infty}}}(\mathcal{U},\mathcal{V})$ has positive upper density, therefore $N_{f_{1,\infty}}(U,V)$ has positive upper density. Hence, $(X,f_{1,\infty})$ is topologically ergodic.
\end{proof}

\section{On Shadowing Property}
In this section, we obtain relation between the shadowing property of the non-autonomous dynamical system $(X,f_{1,\infty})$ and its induced system $(\mathcal{K}(X),\overline{f_{1,\infty}})$.

\begin{thm}
Let $(X,f_{1,\infty})$ be a dynamical system. If $(\mathcal{K}$($X),\overline{f}_{1,\infty}$) has shadowing property, then $(X,f_{1,\infty}$) also has shadowing property.
\end{thm}

\begin{proof}
Suppose $(\mathcal{K}$($X),\overline{f}_{1,\infty}$) has shadowing property. So, every $\delta$-pseudo orbit in $\mathcal{K}(X)$ is $\epsilon$-traced. We need to show that $(X,f_{1,\infty}$) has shadowing property. Let $\gamma=\{x_0,x_1,x_2,\ldots \}$ be a $\delta$-pseudo orbit in $X$. Then $\gamma^*=\{\{x_0\},\{x_1\},\{x_2\},\ldots \}$ is a $\delta$-pseudo orbit in $(\mathcal{K}$($X))$ and therefore by shadowing property of $(\mathcal{K}$($X),\overline{f}_{1,\infty}$), there exists a point $A\in (\mathcal{K}$($X))$ which $\epsilon$-shadows $\gamma^*$, i.e, $d_H(\overline{f}_{0}^{i}(A),\{x_i\})<\epsilon$, for each $i\geq0$. Hence, by definition of Hausdorff metric, we have $d(f_{0}^{i}(a),x_i)<\epsilon$, for each $a\in A$ and for each $i\geq0$. Thus, $\gamma$ is $\epsilon$-shadowed implying that $(X,f_{1,\infty}$) has shadowing property. 
\end{proof}

\begin{lem}
Let $(X,f_{1,\infty})$ be a dynamical system and Y be a dense subset of $X$ such that $Y$ is $f_{n}$-invariant for each $n\geq 1$. If $(Y,f_{1,\infty})$ has finite-shadowing property, then so does $(X,f_{1,\infty})$.
\end{lem}

\begin{proof}
We assume that $(Y,f_{1,\infty})$ has finite-shadowing property. Let $\gamma= \{x_0, x_1, \ldots,x_k\}$ be a $\delta/3$-pseudo orbit in $X$, where $\delta$ is given by shadowing property of $(Y,f_{1,\infty})$ for $\epsilon/2$. As each $f_n$ is continuous and $X$ is compact, therefore each $f_n$ is uniformly continuous for each $n\geq1$. Thus, there exists $\eta>0$ with $\eta<\delta/3$ and $\eta<\epsilon/2$ such that whenever $d(x,y)<\eta$, $d(f_n(x),f_n(y))<\delta/3$. For each i, $0\leq i\leq k$, let $y_{i}\in B_d(x_i,\eta)\cap Y$ then $d(x_i,y_i)<\eta<\delta/3$. Clearly, $\gamma^*=\{y_o, y_1, \ldots, y_k\}$ is a finite $\delta$-pseudo orbit in $Y$. Since $(Y,f_{1,\infty})$ has finite-shadowing property, therefore there exists a point $y\in Y$ which $\epsilon/2$-shadows $\gamma^*$ which implies $d(f_{0}^{i}(y),y_i)<\epsilon/2$, for all i, $0\leq i\leq k$. Hence, $d(f_{0}^{i}(y),x_i)<d(f_{0}^{i}(y),y_i)+d(y_i,x_i)<\epsilon$. Thus, $y$ $\epsilon$-shadows $\gamma$ and we get that $(X,f_{1,\infty})$ has finite-shadowing property.
\end{proof}

\begin{lem}
Let $(X,f_{1,\infty})$ be a dynamical system. If $(X,f_{1,\infty})$ has finite-shadowing property, then it has shadowing property.
\end{lem}

\begin{proof}
Let $\epsilon>0$ and let $\delta$ be given for $\epsilon/2$, by the finite-shadowing property of $(X,f_{1,\infty})$. Let $\{x_n\}_{n\geq0}$ be a $\delta-$pseudo orbit in $X$. For each $n\in \mathbb{N}$, there is a $y_n$ which $\epsilon/2-$ shadows $\{x_0,x_1,\ldots,x_n\}$. Then, $X$ being compact, there is a subsequence $\{y_{n_k}\}_{k\in \mathbb{N}}$ of $\{y_n\}_{n\in \mathbb{N}}$, which has a limit say $y\in X$. So, for any $m\in\mathbb{N}$, there is an $n_k>m$, such that $d(f_{0}^{m}(y_{n_k}),f_{0}^{m}(y))<\epsilon/2$. Therefore, we have $d(f_{0}^{m}(y),x_m)\leq d(f_{0}^{m}(y),f_{0}^{m}(y_{n_k}))+d(f_{0}^{m}(y_{n_k}),x_m)<\epsilon/2<\epsilon$. Hence, $(X,f_{1,\infty})$ has shadowing property.
\end{proof}

\begin{thm}
Let $(X,f_{1,\infty})$ be a dynamical system. If $(X,f_{1,\infty})$ has shadowing property, then $(\mathcal{F}(X),\overline{f_{1,\infty}})$ has finite-shadowing property.
\end{thm}

\begin{proof}
Let $\epsilon>0$ and $\delta>0$ be given by shadowing property of $(X,f_{1,\infty})$. Let $\gamma=\{A_0, A_1,\ldots, A_m\}$ be a finite $\delta$-pseudo orbit in $\mathcal{F}(X)$ and assume that $|A_i|=n_i$, for each i, $0\leq i\leq m$. We will construct a family of $\delta$-pseudo orbits in $X$, denoted by $\{\gamma_j :j\leq n\}$ for some $n$, such that writing $\gamma_j=\{a_{0}^{j}, a_{1}^{j},\ldots, a_{m}^{j}\}$; we have $A_i=\{a_{i}^{j}:j\leq n\}$, for all $i\leq m$. For this, suppose that $A_m=\{a_{m}^{1}, a_{m}^{2},\ldots, a_{m}^{n_m}\}$. For each $j$ with $1\leq j\leq n_m$, we first construct a $\delta$-pseudo orbit in $X$ with $i$-th element in $A_i$ whose final element is $a_{m}^{j}$.

Since $\gamma$ is a $\delta$-pseudo orbit, we can choose $a_{m-1}^{j}\in A_{m-1}$ such that $d(f_m(a_{m-1}^{j}),a_{m}^{j})<\delta$. Again there is some $a_{m-2}^{j}\in A_{m-2}$ such that $d(f_{m-1}(a_{m-2}^{j}),a_{m-1}^{j})<\delta$. Continuing in this way, we have $\delta$-pseudo orbits $\gamma_j=\{a_{o}^{j}, a_{1}^{j},\ldots, a_{m}^{j}\}$, for each $j\leq n_m$ such that $A_m=\{a_{m}^{j}:j\leq n_m\}$ and $\{a_i^j:j\leq n_m\}\subseteq A_i$, for each $i\leq m$.
Let s = max$\{i<m: A_i \neq \{a_i^j:j\leq n_m\}\}$. If no such s exists then we are done, otherwise write $A_s -\{a_{s}^{j}:j\leq n_m\}=\{a_{s}^{j}:n_m<j<n_s'\}$. As done for $A_m$, for each $n_m<j<n_s'$, we construct a $\delta-$pseudo orbit $\gamma_j'=\{a_{o}^{j},\ldots,a_{s}^{j}\}$ such that $a_{i}^{j}\in A_i$ for $i\leq s$ and $A_s=\{a_s^j:j\leq n_s'\}$.

Since $f_{s}(a_s^j)\in \overline{f_{s}}(A_s)$ and $d_H(\overline{f_{s+1}}(A_s),(A_{s+1}))<\delta$, there is an $a_{s+1}^j\in A_{s+1}$ such that $d(f_{s+1}(a_s^j),(a_{s+1}^j))<\delta$. Similarly, for each j, $n_m<j<n_s'$ and for each i, $s<i<m$, $a_{i}^{j}\in A_i$ such that $d(f_{i+1}(a_i^j),(a_{i+1}^j))<\delta$, so we can extend $\gamma_j'$ to a $\delta$-pseudo $\gamma_j$ which starts in $A_0$ and ends in $A_m$. Repeating this, it is clear that we can construct a family $\{\gamma_j:j\leq n\}$ of $\delta$-pseudo orbits in $X$. Since $f$ has shadowing property, for each $\gamma_j$, there exists a point $b_j\in X$ which $\epsilon-$shadows $\gamma_j$. Note that $B=\{b_0,b_1,\ldots,b_k\}$ and $B$ $\epsilon$- shadows $\gamma$. Therefore, $(\mathcal{F}(X),\overline{f_{1,\infty}})$ has finite-shadowing property. 
\end{proof}

\begin{Cor}
Let $(X,f_{1,\infty})$ be a dynamical system, then $(X,f_{1,\infty})$ has shadowing property if and only if $(\mathcal{K}(X),\overline{f_{1,\infty}})$ has shadowing property.
\end{Cor}

\begin{proof}
The proof follows from Lemma 4.1, Lemma 4.2 and Theorem 4.2.
\end{proof}

\section*{Acknowledgement}
The first author is funded by GOVERNMENT OF INDIA, MINISTRY OF SCIENCE and TECHNOLOGY No: DST/INSPIRE Fellowship/[IF160750].

\end{document}